\documentclass{amsart}
\usepackage{amsmath}
\usepackage{amssymb}
\usepackage{amsthm}

\usepackage{comment}
\usepackage[colorlinks]{hyperref}

\newtheorem{theorem}{Theorem}[section]
\newtheorem{lemma}[theorem]{Lemma}
\newtheorem{proposition}[theorem]{Proposition}

\theoremstyle{definition}
\newtheorem{definition}[theorem]{Definition}
\newtheorem{example}[theorem]{Example}

\theoremstyle{remark}
\newtheorem{remark}[theorem]{Remark}

\numberwithin{equation}{section}

\DeclareMathOperator{\sgn}{sgn}
\def\ip{\,\llcorner\,} %shortmid\!\!\!\llcorner}
\def\cL{\mathcal{L}}
\def\bR{\mathbb{R}}
\def\bN{\mathbb{N}}
\def\bT{\mathbb{T}}

\begin{document}

\title[One-dimensional singular parabolic problems with $BV$ data]{Energy solutions to one-dimensional singular parabolic problems with $BV$ data are viscosity solutions}

\author{Atsushi~Nakayasu}%${}^1$}
\address{Graduate School of Mathematical Science, University of Tokyo,
 Komaba 3-8-1, Meguro-ku, Tokyo 153-8914, Japan}
% \curraddr{}
 \email{ankys@ms.u-tokyo.ac.jp}

\author{Piotr~Rybka}%${}^2$} 
\address{Institute of Applied Mathematics and Mechanics,
Warsaw University\\ ul. Banacha 2, 07-097 Warsaw, Poland}
\email{rybka@mimuw.edu.pl}

% \keywords{}

\date{\today}

%\dedicatory{}

%\commby{}

\begin{abstract}
We study one-dimensional very singular parabolic equations with periodic boundary conditions and initial data in $BV$, which is the energy space. We show existence of solutions in this energy space and then we prove that they are viscosity solutions in the sense of Giga-Giga.
\end{abstract}

\maketitle



\section{Introduction}
\label{s:intro}

We study a Cauchy problem for one-dimensional singular parabolic equations of the form
\begin{eqnarray}
\label{e:ge}
& \partial_t u = \left(\cL(u_x)\right)_x, & (x,t)\in Q_T: =\bT\times (0,T), \\
\label{e:ic}
& u(x, 0) = u_0(x), & x\in \bT,
 \end{eqnarray}
where
$\cL:\bR\to\bR$ is monotone increasing and bounded. Such equations arise in models of crystal growth or image analysis. This is why it is reasonable to assume that $u_0\in BV$. For the sake of simplicity, we consider periodic boundary conditions. Here, $\bT$ is the one-dimensional flat torus $\bT  = \bR /\mathbb{Z}$.

Since $\cL$ is merely monotone increasing, it may have jumps. Typical examples are
\begin{equation}
\label{e:tvf}
\partial_t u = (\sgn(u_x))_x,
\end{equation}
and
\begin{equation}
\label{e:tvf2}
\partial_t u = (\sgn(u_x-1)+\sgn(u_x+1))_x .
\end{equation}
Equation (\ref{e:tvf}) is a prototype for the class, we have in mind, (\ref{e:ge}).
%because it is a gradient flow of the total variation. % as we  explain later.
We also point out that equation \eqref{e:tvf} can be formally written as $\partial_t u = 2\delta(u_x)u_{xx}$,
where $\delta$ is the Dirac delta function.
%Such parabolic equations whose diffusion coefficient contains a non-local term are called %\emph{very singular diffusion equations} or \emph{sudden directional diffusion}.

Since problem (\ref{e:ge}) is one-dimensional, 
then for a given increasing $\cL$, we can always find convex $W$,  which we will call the energy density such that
$$
\cL (p) = \frac{dW}{dp}(p),\qquad \hbox{for } a.e.\ p\in \bR.
$$
Thus,  we can write  (\ref{e:ge}) %and formally  (\ref{e:tvf2}) 
as  an $L^2$-gradient flow
\begin{equation}\label{e:fgf}
\partial_t u = -\nabla E[u]
\end{equation}
with the corresponding energy, defined by
\begin{equation}\label{d:E}
E[u] = 
\left\{
\begin{array}{ll}
 \int_{\bT} W( Du) & \hbox{if }u \in BV(\bT),\\
 +\infty& \hbox{if }u \in L^2(\bT) \setminus  BV(\bT).
\end{array}
\right.
\end{equation}
In order to avoid unnecessary technicalities at this stage, we will restrict our attention to piecewise linear $W$'s which are coercive, i.e. $W(p)\to\infty$ when $|p|\to\infty$. We shall see (cf. (\ref{deWi})) that indeed
$E[u]$ is well-defined. We will look for solutions for which $t\mapsto E[u(t)]$ is bounded over $[0,T)$, we will call them {\it energy solutions}. There is  substantial literature on functionals on measures, including \cite{valadier}, \cite{Mandallena} and others.  The general case of $W$ requires additional considerations, see \cite{AmPa}, \cite{valadier}, \cite{fonseca}, \cite{Mandallena}.

In case of (\ref{e:tvf}),  we have $W(p) = |p|$, this is why  it
is called the \emph{total variation flow}.
Equation \eqref{e:tvf} has been studied quite extensively, \cite{fukui},  \cite{GGK01}, \cite{GiGoRy}, \cite{kiemury}, \cite{Mucha}, \cite{mury}, \cite{spohn}. These authors did not allow jumps in the initial datum $u_0$. For the sake of completeness, we should mention the growing literature on the higher dimensional versions of \eqref{e:tvf}. Let us mention only \cite{dirichlet.andreu}, \cite{briani-novaga}, \cite{med1}, \cite{med-book}, \cite{BeCaNo}, \cite{GGP13}.

On the other hand, (\ref{e:tvf2}) was studied in \cite{Miry} with initial condition $u_0\in BV$. There, the showed that if $u_0\in BV$, then the unique solution $u$, which we constructed there, belongs to $L^\infty(0,T; BV)$, see \cite[Theorem 2.1]{Miry}. Moreover, this kind of regularity is optimal, because there are discontinuous solutions. 
%Peculiarity of (\ref{e:tvf2}) is existence of two, possibly completing singular slopes, $\pm 1$, %this aspect of (\ref{e:tvf2}) was  studied .

In case of   (\ref{e:tvf2}), we have  $W (p)= |p+1| + |p-1|$
%We showed in \cite{Miry} that the optimal regularity for  (\ref{e:tvf2}) with data in $BV$ is this %space.  We can also call these solutions energy solutions, because if $E(u_0)<\infty$, then %$E(u(t))<\infty$ for almost all $t\in(0,T)$. 
then, we define $E[u]$ for  $u\in BV$ in a natural way, introduced in \cite{Miry}, which can be easily extended to convex piecewise linear $W$. 
However, for the sake of simplifying the exposition, we will use a piecewise linear structure of $W$, generalizing the method used in \cite{Miry}.
A difficulty here is the composition of nonlinear function $\cL$ with a measure $Du$, when $u\in BV$.

We also showed in \cite{Miry} that energy solutions are viscosity solutions in the sense of \cite{GGR14}. However, the proof of this fact presented in \cite{Miry} was direct and involved.
Here, we would like to generalize this fact. 
However, in order  to develop  intuition we will  restrict our attention to the case of convex, piecewise linear and coercive $W$. This will be a necessary preparation before tackling the general case of $W$ with linear growth. The direct method of (\cite{Miry}) is not likely to  work.

We will show that problem (\ref{e:ge}--\ref{e:ic}) is well posed for $u_0\in BV$, which is the energy space. We will exploit for this purpose the gradient flow structure of  (\ref{e:ge}).
%we can regard this equation as
%\begin{equation}\label{e:gf}
%\partial_t u \in - \d E(u) .
%\end{equation}
%where $\d  E(u)$ denotes the subdifferential of $E$. In view of our formal definition of solution %eq. (\ref{e:gf}) is more precise than  (\ref{e:fgf}).
%
%In particular, we have to show that this approach is correct, i.e. we may evaluate $E(u)$ when $\in BV$.  
We will construct solutions by means of regularizations as in 
\cite{mury} and \cite{Miry}. This will be done in Theorem \ref{tw-gl} below. However, our main result is Theorem \ref{t:ma} stating that energy solutions to (\ref{e:ge}) are also viscosity solutions in the sense of \cite{GGR14}. With this result, we obtain new tools to study properties of solutions to (\ref{e:ge}). Incidentally, this theorem shows that  for regular data, i.e. when not only $u_0$ is in  $BV$ but also $u_{0,x}\in BV$, then the notion of energy solutions, the almost classical solutions introduced in \cite{mury2} and the viscosity solutions to (\ref{e:ge}) coincide.

The rational behind our present work is the usefulness of viscosity solutions.
Originally, the 
viscosity solution theory was introduced in early 1980s by Crandall and Lions, \cite{Cral} as a weak solution of first-order Hamilton-Jacobi equation
and generalized to some regular second-order equations around 1990;
see \cite{CIL92}. A big success of this theory was the study of the  mean curvature flow equation independently
by Evans-Spruck \cite{ES91} and Chen-Giga-Goto \cite{CGG91}.
The viscosity solution for the very singular diffusion equation is first introduced by Giga-Giga \cite{GG98,GG99, GG01}. 
The authors establish a theorem on a unique existence of the viscosity solution of a class of very singular diffusion equations including \eqref{e:ge} with continuous initial data $u_0$. The present paper provides more of examples existence theorems, this time with discontinuous initial conditions.

The plan of the paper is as follows: in Section 2, we show a lower semicontinuity of functional $E$, then existence of solutions in the energy space by the regularization method. Section 3 is devoted to introducing the necessary notions pertinent to the viscosity solutions. In Section 4, we construct monotone sequences of continuous functions approximating the initial conditions. This guarantees the existence of continuous solutions to (\ref{e:ge})--(\ref{e:ic}). Then, passing to the limit with the regularizing parameter yields the desired result.
%In view of this fact, the authors of \cite{Miry} prove that a variational solution is a viscosity solution for the equation \eqref{e:tvf2} with $BV$ initial data.
%However, thier argument relies the form of the equation
% and seems to be applied only to piecewise liniear $W$.
%and cannot be applied to the generalized equation in which $W$ is not piecewise linear.

\section{Energy solutions}
\label{s:pre}
We  could use the abstract semigroup theory, see \cite{brezis}, exploiting the gradient flow structure, (\ref{e:fgf}), to prove existence and uniqueness of solutions to (\ref{e:ic}) with data in $L^2$ or $BV$. For this purpose, it suffices to show  that $E$ defined by (\ref{d:E}) is convex, proper and lower semicontinuous with respect to the $L^2$ topology. However, in this paper, not only the fact that $u_0$ belongs to $BV$ is important but also how $u_0$ is approximated by smooth functions. This is why we will use the method of regularization to construct energy solutions. By `energy solutions' we will understand weak solutions with data from the vector space where $E$ is finite. We restrict our attention 
to $W$ which are piecewise linear, convex and coercive.

\subsection{Functional $E$}

Understanding $E$ is crucial for the present paper.
We consider only $W$ given by
\begin{equation}\label{deW}
 W(p) = \sum_{j=1}^{N^+} \alpha_j^+ (p+\beta_j^+)^+ 
 + \sum_{j=1}^{N^-} \alpha_j^- (p+\beta_j^-)^- ,
\end{equation}
where  $N^\pm$ are positive integers, all $\alpha^\pm
_i$ are positive and $\beta^\pm_i$ are real. We adopt the usual definition,
$x^+ = \max\{0,x\}$ and $x^- = (-x)^+$ for a real number $x$. 

%In priciple at most one of $N^+$ and $N^-$ may be zero, but we postpone this case for future. %\marginpar{COME BACK}

First, we will explain that these are all functions of interest to us.

\begin{proposition}
Let us suppose that $W:\bR\to\bR$ is piecewise linear, convex and there is $M>0$ such that $\frac{dW}{dp}(p)$ in constant for $|p|>M$.
%$\lim_{|p|\to +\infty}W(p) =+\infty$. 
Then, $W$ has the form (\ref{deW}). 
\end{proposition}

\begin{proof}
We will first find $ \beta_i^\pm$'s. We notice that coercivity and convexity of $W$ imply that
$$
a_0:= \min \{x:\ W(x) = \min\{ W(p):\ p\in \bR\}\}
$$
is well defined.
Since by assumption $W$ is linear for large arguments, then by definition,
there are real $a_i$, $i= -M, \ldots, N$, (we also set $a_{N+1} = +\infty$) such that
$$
\begin{array}{ll}
 W(p) = s_i(x-a_i) + W(a_i), &\hbox{for } x\in [a_i,a_{i+1}),\quad i= -M, \ldots, N+1,\\
 W(p) = s_{-M-1} (x- a_{-M}) W(a_{-M}) &\hbox{for } x\in (-\infty, a_{-M}).
\end{array}
$$
If $s_0 =0$, then we set 
\begin{eqnarray*}
\beta_i^+ = a_i, \quad i=1,\ldots, N,& \hbox{and }&N^+:=N,\\
\beta_i^- = a_{1-i}, \quad i=1,\ldots , M, &\hbox{and }&N^-:=M.
\end{eqnarray*}
Otherwise, we define
\begin{eqnarray*}
\beta_i^+ = a_{i-1}, \quad i=1,\ldots, N+1, &\hbox{and }&N^+:=N+1,\\
\beta_i^- = a_{1-i}, \quad i=1,\ldots , M, &\hbox{and }&N^-:=M.
\end{eqnarray*}

Monotonicity of the derivative of $W$ implies monotonicity of sequence $s_i$,  $i= -M, \ldots, N$.
We set $\alpha_1^+ = s_1$. Since $s_{i+1}>s_i$, then there is positive $\alpha_{i+1}$ such that
$$
s_{i+1} = \alpha_{i+1}^+ + s_i = \sum_{k=1}^i \alpha_k^+, \qquad i =1,\ldots, N^+-1.
$$
Similarly, we define $\alpha_i^-$, $i =1,\ldots, N^-$.
\end{proof}

Formula (\ref{deW}) and positivity of $N^+$ and $N^-$ imply the existence of positive $c_0, c_1, c_2$ such that
\begin{equation}\label{rngrw}
 c_1|p| - c_2 \le W(p) \le c_0(|p|+1).
\end{equation}

The advantage of (\ref{deW}) is clearly seen when it comes to defining $\int_\bT W(Du)$, because it refers to well-defined operations on measures. Indeed, if $\mu$ is a (finite) signed measure, then $\mu^\pm = \frac12(|\mu|\pm\mu)$, where $|\mu|$ is the variation of $\mu$.
%, needed at this stage? AN feel $u+ a x$ is ambiguous}
%Moreover, if $u\in BV(\bT)$, then $u+ a x\in  BV(\bT)$ too.
Thus, for  $u\in BV(\bT)$, the following formula is correct, despite the lack of homogeneity of $W$,
\begin{equation}
\label{deWi}
 \int_\bT W(Du) 
 = \int_\bT \left( \sum_{i=1}^{N^+} \alpha_i^+ (Du+\beta_i^+)^+ 
 + \sum_{j=1}^{N^-} \alpha_j^- (Du+\beta_i^-)^- \right).
\end{equation}
%Here, we use the fact that We conclude that that functional $E$ given by (\ref{d:E}) is well-defined.

We notice that obviously $E$
is convex. It is also lower semicontinuous.

\begin{proposition}\label{p:2.2}
 $E$ is lower semicontinuous with respect to the convergence in the $L^2$ topology.
\end{proposition}

\begin{proof}
Let us suppose that  $u_n \to u$ in $L^2$. Since for any $v\in BV(\bT)$  and $a \in \mathbb{R}$, we have
$$
\int_\bT (Dv +a) = a | \bT|,
$$
then in order to prove our claim, it is 
sufficient to check  that
$$
\varliminf_{n\to+\infty} \int_{\bT}|Du_n + a| \ge \int_{\bT} |Du + a|.
$$

We may assume $\int_\bT |Du_n| \le M < \infty$, $n \in \bN$, for otherwise there is nothing to prove.
This bound and the lower semicontinuity of the $BV$ norm imply
$$
\varliminf_{n\to+\infty} \int_{\bT}|Du_n | \ge \int_{\bT} |Du|,
$$
as a result, $u \in BV$.

We recall that if $u\in BV(\bT)$, then 
$$
Du = \frac{du}{dx} \ip \cL^1 + D^c u + \sum_{i\in J} a_i \delta_{x_i},
$$
where $\frac{du}{dx} \ip \cL^1$ is  absolutely continuous with respect to $\cL^1$ part of $Du$, $D^c u$ is the continuous part singular with respect to $\cL^1$ and $ \sum_{i\in J} a_i \delta_{x_i}$ is the jump part.
Since the number of jump discontinuities points of $u$ and $u_n$, $n\in \bN$, is at most countable, then
%We notice that 
there exists $x_0\in \bT$, a common 
continuity point of all $u_n$, $n\in \bN$ and $u$. That is,
\begin{equation}\label{e:co}
 |D u_n|(\{x_0\}) = 0 = |D u|(\{x_0\})\qquad \hbox{for all }n\in \bN.
\end{equation}
Once we identify $\bT$ with $[0,1)$, then for the sake of simplicity of notation we may assume that $x_0=0$. We notice that
$$
Du_n + a %\beta_i^\pm  
= D(u_n + a%\beta_i^\pm x
)  - a %\beta_i^\pm 
\delta_0.
$$
Furthermore, due to (\ref{e:co}), we have
$$
| D(u_n + a%\beta_i^\pm x
)  - a %\beta_i^\pm 
\delta_0 | = 
| D(u_n + a %\beta_i^\pm 
x)| \ip \bT\setminus \{0\}.
$$
We notice that $(u_n + a%\beta_i^\pm 
x) \chi_{\bT\setminus\{0\}}$ is a bounded sequence in $BV(\bT)$,
hence the lower semicontinuity of the norm implies that
$$
\varliminf_{n\to+\infty} \int_{\bT}|D(u_n + a%\beta_i^\pm 
x )| \ge \int_{\bT} |D(u + a %\beta_i^\pm 
x)|. 
$$
\end{proof}

\begin{remark}
Density $W$ is not homogeneous, but since $W$ is piecewise linear, one could easily see that  $E$ coincides with its  relaxation, cf. \cite{AmPa}, \cite{fonseca}, \cite{Mandallena}.
\end{remark}

\subsection{Existence of energy solutions by the regularization method}

In this section our goal is to examine the existence of solutions to \eqref{e:ge}, \eqref{e:ic},
when the nonlinearity  $\cL$ is the derivative of a convex function $W$, given by (\ref{deW}). 
We studied problem (\ref{e:ge}--\ref{e:ic})  in \cite{mury},  for rather general nonlinearities $\cL$ and Dirichlet boundary conditions, while assuming that $u_0\in L^1$ and $u_{0,x}\in BV$. In \cite{Miry} we considered general data $u_0\in BV$, while focusing on  special $\cL$,
$$
\cL(p) = \sgn(p+1) + \sgn(p-1).
$$
Here, we will use the method of \cite{mury} and \cite{Miry} to prove the existence of energy solutions.

Of course, one could study $W$ with a different growth rate at infinity, but we will not do this here. We would be happy to work with general $W$, but for the time being we will consider only $W$ which is convex and piecewise linear, given by (\ref{deW}).
In principle,  one of $N^+$ and $N^-$ could be zero, but we will not address this issue here.

We introduce the definition of solutions to (\ref{e:ge}) with 
periodic boundary conditions. In principle, it is well-known, see
\cite{dirichlet.andreu}, \cite{mury}.

\begin{definition}
\label{d1}
We shall say that  function $u\in L^2(0,T;L^2(\bT))$ is a {\it finite energy} solution to %\marginpar{zgodne z Mazonem?}
(\ref{e:ge}) if $u\in L^\infty(0,T;BV(\bT))$ and  $u_t\in L^2(0,T;L^2(\bT))$ and there is $\Omega \in L^2(0,T; W^{1,2}(\bT))$ satisfying the identity
\begin{equation}
\label{rn-d1}
\langle u_t, \varphi \rangle = - \int_\bT \Omega \varphi_x\,dx
\end{equation}
for all test functions $\varphi \in C^\infty(\bT)$ and for almost every $t>0$.
\end{definition}

We notice that the time regularity, postulated in Definition \ref{d1}
implies that solutions to  (\ref{e:ge}) are in $C([0,T];L^2(\bT))$. Hence,
we can impose initial conditions (\ref{e:ic}).
At this point we stress that $\Omega$ is a selection of the composition of 
multivalued operators $\cL\circ u_x$.

We will show the existence result. We note that we consider less regular initial conditions than in \cite{mury}. We prefer to use tools based on approximation of (\ref{e:ge}) by smooth problems. 

\begin{theorem}
\label{tw-gl}
Let us suppose that $u_0\in BV(\bT)$, then there exists an energy solution to (\ref{e:ge}-\ref{e:ic}).
Moreover, %$u_t\in L^2(0,T;L^2(I)),\ u\in L^\infty(0,T;BV(I))$ and 
since for almost all $t>0$, $u$ satisfies
\begin{equation}
\label{rn-tw1}
\int_\bT [W(u_x + h_x) - W(u_x)]\,dx \ge \int_\bT \Omega  h_x\,dx,
\end{equation}
for all $h\in C^{\infty}(\bT)$,
hence $u$ is a unique solution.
\end{theorem}

\begin{proof}
\textit{Step 1.}
After regularizing $\cL$ and $u_0$, we obtain a uniformly parabolic problem,
\begin{equation} \label{rnep}
 \begin{array}{ll}
  \frac{\partial u^\epsilon}{\partial t}= \left( {\cL^\epsilon}(u^\epsilon_x)\right) _x, & (x,t)\in Q_T,\\
u^\epsilon(x,0)= u^\epsilon_0(x), & x\in \bT,
 \end{array}
\end{equation}
where $\epsilon$ is a regularizing parameter. We define
$W^\epsilon$ by the formula below,
\begin{equation}\label{rwe}
 W^\epsilon (p) = \int_0^p \cL^\epsilon(s)\,ds + \frac{\epsilon}{2} p^2. 
\end{equation}
We notice that 
\begin{equation}\label{2.8pol}
W(p) \le W^\epsilon (p) \le W(p) + k \epsilon + \frac{\epsilon}{2} p^2, 
\end{equation}
where $k$ is chosen in an appropriate way.

By the classical theory, see
\cite{lady}, we obtain existence and uniqueness of smooth solutions to
(\ref{rnep}). 

If we multiply (\ref{rnep}) by $u^\epsilon_t$ and integrate over $Q_T$, then we reach,
$$
\int_0^T\int_{\bT} (u^\epsilon_t)^2\, dxdt = 
\int_0^T\int_{\bT} ( {\cL^\epsilon}(u^\epsilon_x))_x u^\epsilon_t\, dxdt.
$$
Integration by parts yields,
$$
\int_0^T\int_\bT (u^\epsilon_t)^2\, dxdt = - \int_0^T\int_\bT  {\cL^\epsilon}(u^\epsilon_x) u^\epsilon_{xt}\, dxdt = - \int_0^T\int_\bT \frac{d}{dt} W^\epsilon(u^\epsilon_x)\, dxdt.
$$
We identify $\bT$ with the interval $[0,1)$. We have to justify, why the boundary terms drop out in the formula above. For this purpose, we recall a 
%Because of periodicity and the following 
well-known result, see e.g. \cite{Loj}.

\begin{lemma}
\label{cal-cze}
If $w\in H^1(a,b)$ and $\varphi\in BV(a,b)$, then
$$
\int_a^b w_x\varphi dx = -\int_a^b wD\varphi + \gamma(w\varphi)|_a^b. \eqno\Box
$$
\end{lemma}

Due to the periodicity of the ingredients, the integral over $(a,b)$ equals the integral over $(a+\delta,b+\delta)$ and we can select such $\delta$ that
$x=a+\delta,$ $x=b +\delta$ are points of continuity of $u -v$, as considered over $\bR$.

Hence, we reach the following conclusion, 
\begin{equation}
 \label{rnee}
 \int_0^T\int_\bT (u^\epsilon_t)^2\, dxdt + \int_\bT W^\epsilon(u^\epsilon_x(x,T))\, dx =
 \int_\bT W^\epsilon(u^\epsilon_{0,x})\, dx.
\end{equation}
Now, we will pass to the limit.
First of all, we notice that (\ref{2.8pol}) implies
$$
\int_\bT W^{\epsilon}(u^{\epsilon}_{0,x})
\le k\epsilon |\bT| + %B_1 
\sup_{\epsilon\in[0,1]} \int_\bT W(u^{\epsilon}_{0,x})+\int_\bT \frac{\epsilon}{2} |u^{\epsilon}_{0,x}|^2\,dx
=: M < +\infty,
$$
where we used the following inequality,
$$
\frac{\epsilon}{2}  \int_\bT|u^{\epsilon}_{0,x}|^2\,dx \le C \| u_0\|_{BV}^2.
$$
Since we have found a bound on the right-hand-side (RHS) of (\ref{rnee})
independent of $\epsilon$, we conclude that
$$
\int_0^T\int_\bT(u_t^\epsilon)^2\le M
\quad \text{and}
\quad {\sup_{t \in [0, T]}}\int_\bT W^\epsilon(u^\epsilon_x(x,t))\le M.
%\qquad \text{\textcolor{red}{for all $t \in [0, T]$.}}
$$ 

%\textcolor{red}{
Thus, we can select a subsequence $\{u^\epsilon\}$ such that
%The weak lower semicontinuity of the norm implies that if 
$$
u^\epsilon\rightharpoonup u\quad \textrm{in}\quad L^2(0,T;L^2(\bT))
$$
and 
$$
u^\epsilon_t\rightharpoonup u_t\quad \textrm{in}\quad L^2(0,T;L^2(\bT)).
$$
Moreover, if we set $B = W^{1,1}(\bT)$, $X=Y =L^2(\bT)$, then the above estimates show that
family $\{u^\epsilon\}$ is bounded in $L^\infty(0,T,B) \cap L^1(0,T;X)$ and 
$\{\frac{\partial u^\epsilon}{\partial t}\}$ is bounded in $L^1(0,T;Y)$.
% then by Aubin Lemma, see \cite[Corollary 6]{simon} we deduce that ...
%Thus, the family $\{ u^\epsilon \}$ is included in a bounded ball of the 
%$$
%\{ u \in L^\infty(0, T; BV(\bT)) \mid u_t \in L^2(0, T; L^2(\bT)) \}.
%$$
Since $BV$ is compactly embedded in $L^2$, then
by Aubin Lemma (see \cite[Corollary 6]{simon}),
we can select a subsequence $\{u^\epsilon\}$ such that
$$
\text{$u^\epsilon \to u$ in $L^p(0,T,L^q(\bT))$,}
$$
where $p,q$ are arbitrary from the interval $(1, \infty)$.
As a result for almost all $t\in(0,T)$
$$\|u^{\epsilon}(\cdot,t)-u(\cdot,t)\|_{L^q} \to 0.$$
Thus, we use the lower semicontinuity of $E$, %the $BV$ norm 
see Proposition \ref{p:2.2}, 
we deduce that for almost all $t>0$
$$\varliminf_{\epsilon \to 0}\int_\bT W^{\epsilon}(u_{x}^{\epsilon})(x,t)dx \ge \varliminf_{\epsilon \to 0}\int_\bT W(u_{x}^{\epsilon})dx\ge \int_\bT W(Du)(\cdot,t).$$ 
Combining these inequalities, we arrive at
\begin{equation} \label{rne2}
 \int_0^t\int_\bT u _t^2(x,s)\, dxds + \int_\bT W(Du)(\cdot,t))\, dx \le M\quad \textrm{for almost all}\ t\in(0,T).
\end{equation}
%We note that (\ref{rne2}) does not involve any statement on the boundary values of $u$.

Moreover, due to (\ref{rngrw}), we have a bound on the $BV$ norm of $u(\cdot,t)$, %\marginpar{TO DO}
\begin{eqnarray*}
 \displaystyle{\int_\bT|Du|} \le C (1+ \int_\bT W(Du)).
\end{eqnarray*}

We also have to indicate a candidate for $\Omega$, as required by the definition
of a solution. We set 
$$
\Omega^\epsilon(x,t) := \cL^\epsilon(u^\epsilon(x,t)).
$$
Since $u_t^\epsilon = \Omega^\epsilon_x$, then due to (\ref{rnee}), we deduce that
\begin{equation}\label{rnom}
 \| \Omega^\epsilon \|_{L^2(0,T;H^1(\bT))} \le M_1<+\infty.
\end{equation}
Hence, we can select a  subsequence,
$$
\Omega^\epsilon\rightharpoonup \Omega\quad \textrm{in}\quad L^2(0,T;H^1(\bT)).
$$
Moreover,
$$
\int_0^T\int_\bT u_t \varphi\,dxdt = \int_0^T\int_\bT \Omega_x \varphi\,dxdt
\quad\hbox{for all } \varphi\in C_0^\infty((0,T)\times I).
$$
At this point, we may apply \cite[Lemma 2.1]{mury} to conclude that
(\ref{rn-d1}) holds. Moreover, \cite[Lemma 2.2]{mury} implies (\ref{rn-tw1}).

\textit{Step 2.}
We shall establish uniqueness of solutions. 
We notice that if $u$ is a solution to (\ref{e:ge}), according to Definition \ref{d1}, then $t\mapsto u(t)\in L^2(\bT)$ is continuous. In particular it makes sense to evaluate $u$ at $t=0$.

Let us suppose that $u$ and $v$ are solutions to (\ref{e:ge}) satisfying
%(\ref{rnD}) with 
$u(0)=u_0 = v(0)$. %\marginpar{how to integrate nicely?}
Since we identify $\bT$ with $[0,1)$, we notice that for any 
$\delta\in(0,1)$. We have 
\begin{equation}\label{mnm}
\int_{0}^1\left({\Omega}_x({u}) -{\Omega}_x({v})\right)({u}-{v})\,dx = 
\int_{\delta}^{1+\delta}\left({\Omega}_x({u}) -{\Omega}_x({v})\right)({u}-{v})\,dx.
\end{equation}
Moreover, ${\Omega}({u}) -{\Omega}({v})$ belongs to $H^1( \bT)$, while
${u}-{v}$ is a function from $BV(\bT)$. We want to perform
integration  by parts in the RHS of (\ref{mnm}). For this purpose, we use Lemma \ref{cal-cze} with $\delta>0$ so chosen that the boundary terms drop out due to continuity of the integrand at $\delta$.
%to conclude that ${u}-{v}$.
%If we do so, we notice 
Thus, we conclude that %with the help of 
\begin{eqnarray*}
 \frac12 \|{u}-{v}\|_{L^2({I})}^2(T) &= 
\displaystyle{-\int_0^T\int_{{I}+\delta} ({\Omega}({u}; x,t)-{\Omega}({v}; x,t))({u}_x-{v}_x)\,dxdt}.
%\\& +
%\displaystyle{\int_0^T(\gamma\Omega(u; x,t)-\gamma\Omega(v; x,t))(u-v)|_a^b\,dt.}
\end{eqnarray*}
On the other hand, monotonicity of $\cL$ yields
\begin{equation}\label{rn3}
 \frac12 \|{u}-{v}\|_{L^2(\bT)}^2(\tau) \le 0. 
%\int_0^\tau(\gamma\Omega(u; x,t)-\gamma\Omega(v; x,t))(u-v)|_a^b\,dt.
\end{equation}
We conclude that $u=v$, as desired.
\end{proof}

\subsection{Regularity of energy solutions}

We recall the following fact, which will be very useful for us.

\begin{proposition}[\cite{mury}]
\label{t:re}
 Let us suppose that $u_0\in BV$ as well as $u_{0,x}\in BV$. Then there is $\alpha>0$ such that
 $u\in C^{2\alpha, \alpha}$.
\end{proposition}

We refer the reader to \cite[formula (2.19)]{mury}.

\subsection{Stability of energy solutions}

We show a very useful result for further studies.

\begin{proposition}\label{t:sta}
Let  us suppose that $(u^n,\Omega^n) $ is a sequence of  energy  solutions to (\ref{e:ge}--\ref{e:ic}),
which is bounded in $L^\infty$,
$u_n $ converges to $u$ almost everywhere and $\Omega^n \rightharpoonup \Omega$ in $L^2(0,T;H^1(I))$. Then, $u$ and $\Omega$ form an energy  solution to (\ref{e:ge}--\ref{e:ic}) with initial data $u_0$.
\end{proposition}

\begin{proof}
 Essentially, this follows from our definition of energy solutions and \cite[Lemma 2.1]{mury}.
\end{proof}

%We worry about the passage to the limit in
%$$
%\lim_{n\to \infty}\int_\bT \phi \Omega_{n,x} = 
%$$
%\todo{justify!}

\section{Viscosity solutions}
\label{s:visc}

In this section, we review the viscosity solution theory to the singular diffusion equation of the generalized form
\begin{equation}
\label{e:gsde}
\partial_t u+F((W'(u_x))_x) = 0 \quad \text{in $U_T := (0, T)\times U$}.
\end{equation}
Here, $U$ is an open set in $\mathbb{T}$ and $F$ is a function such that
\begin{itemize}
\item[(F1)]
(Continuity)
$F \in \mathit{C}(\bR)$,
\item[(F2)]
(Ellipticity)
$F$ is non-increasing.
\end{itemize}
We refer the readers to the papers \cite{GGR14, GGN13} for the details.

Let $\bar{\mathbb{R}}$ denote extended real numbers $\mathbb{R}\cup\{ \pm\infty \}$,
and for an $\bar{\mathbb{R}}$-valued function $u$ on $U_T$ define its upper (resp.\ lower) semicontinuous envelope $u^*$ (resp. $u_*$) by
$$
\text{$u^*(t, x) := \varlimsup_{(s, y) \to (t, x)}u(s, y)$\qquad (resp.\ $u_*(t, x) := \varliminf_{(s, y) \to (t, x)}u(s, y)$).}
$$

\subsection{The notion of viscosity solutions to (\ref{e:ge})}

We recall a notion of faceted functions. Let 
$P$ denote the set of the jump points of the derivative of $W$.
A function $f \in \mathit{C}^{1}(U)$ is \emph{faceted} at a point $\hat{x} \in U$ with \emph{slope} $p \in \mathbb{R}$ (or $p$-\emph{faceted} at $\hat{x}$)
if there exists a closed nontrivial finite interval $I = [c_{l}, c_{r}] \subset U$ containing $\hat{x}$ (i.e.\ $c_{l}, c_{r} \in U$ satisfy $c_{l} < c_{r}$ and $c_{l} \leq \hat{x} \leq c_{r}$) such that
\begin{align*}
f'(x) = p &\quad \text{for all $x \in I$,} \\
f'(x) \neq p &\quad \text{for all $x \in J\setminus I$}
\end{align*}
for a neighborhood $J = (b_l, b_r)\subset U$ of $I$.
The closed interval $I$ is called a \emph{faceted region} of $f$ containing $\hat{x}$.
Let $C^2_P(U)$ denote the set of all $f \in C^2(U)$ such that $f$ is $p$-faceted at $\hat{x}$ whenever $p = f'(\hat{x}) \in P$.
We also define the \emph{left transition number} $\chi_{l} = \chi_{l}(f, \hat{x})$ and the \emph{right transition number} $\chi_{r} = \chi_{r}(f, \hat{x})$ for a $p$-faceted function $f$ at $\hat{x}$ by
\begin{align*}
\chi_{l} &=
\begin{cases}
+1 &\text{if $f' < p$ on $(b_{l}, c_{l})$,} \\
-1 &\text{if $f' > p$ on $(b_{l}, c_{l})$,} \\
\end{cases}
\quad &
\chi_{r} =
\begin{cases}
+1 &\text{if $f' > p$ on $(c_{r}, b_{r})$,} \\
-1 &\text{if $f' < p$ on $(c_{r}, b_{r})$.} \\
\end{cases}
\end{align*}

Now, for $\Delta > 0$, $I %= [c_l, c_r] 
\subset U$, $\xi_l, \xi_r \in [-1, 1]$ and $Z \in C^{1, 1}(I)$, we consider the obstacle problem to minimize
$$
J[\xi] :=
\begin{cases}
\int_{I}|\xi'(x)|^{2}\mathit{d}x &\text{if $\xi \in H^1(I)$,} \\
\infty &\text{if $\xi \in L^2(I)\setminus H^1(I)$}
\end{cases}
$$
over
$$
K = \{ \xi \in \mathit{H}^{1}(I) \mid \text{$ \xi -Z\in[ -\frac\Delta2, \frac\Delta2]$ on $I$, 
$\xi(c_{i}) = Z(c_{i}) + (-1)^{k(i)}\chi_{i}\frac\Delta2$%, $\xi(c_{r}) = Z(c_{r})+\chi_{r}\frac\Delta2$
} \},
$$
where $i=l,r$ and $k(l) = 1$, $k(r) = 0$.

We note that it is easy to see that if the function $Z$ is affine, the minimizer is affine as well.
This is the case considered in this paper but, in general, even if $Z \in C^{1, 1}(I)$, then it is known that the unique minimizer belongs to $Z \in C^{1, 1}(I)$;
see \cite{GGR14}.
Also, we note that the derivative of  minimizer $\bar{\xi}$ with $C^1$ smoothness is invariant with respect to shifts of  $Z$ by 
an additive constant.
In view of these remarks, we write
$$
\Lambda^{Z'}_{\chi_{l}\chi_{r}}(x; I, \Delta) = \bar{\xi}'(x) \quad \text{for $x \in I$.}
$$

For $f \in \mathit{C}^{2}_{P}(U)$ and $\hat{x} \in U$,
we define the \emph{nonlocal curvature} $\Lambda_{W}(f)(\hat{x})$ as below.
On the one hand, if $f'(\hat{x}) \notin P$,
we set
$$
\Lambda_{W}(f)(\hat{x}) = W''(f'(\hat{x}))f''(\hat{x})
$$
as expected.
On the other hand, if $p := f'(\hat{x}) \in P$, i.e.\ $f$ is $p$-faceted at $\hat{x}$,
then we set
$$
\Lambda_{W}(f)(\hat{x}) = \Lambda^{\sigma}_{\chi_{l}\chi_{r}}(\hat{x}; I, \Delta)
$$
where $\Delta$ is the jump in the derivative of $W$ at $p$, i.e. 
$ \Delta= \frac{d^+ W}{dq}(p) -  \frac{d^- W}{dq}(p)$, $I = R(f, \hat{x})$, $\chi_{l} = \chi_{l}(f, \hat{x})$, $\chi_{r} = \chi_{r}(f, \hat{x})$, $\sigma(x) = 0$.

Let us denote, by $\mathit{A}_P(U_T)$, the set of all \emph{admissible} functions on $U_T$,
i.e.\ functions $\varphi$ of the form
$
% \label{e:admisfunc@5}
\varphi(t, x) = f(x)+g(t)
$
where $f \in \mathit{C}^{2}_{P}(U)$ and $g \in \mathit{C}^{1}(0, T)$.

\begin{definition}[Viscosity solutions]
\label{d:v}
Let $u$ be an $\bar{\mathbb{R}}$-valued function on $U_T$.
We say that $u$ is a \emph{viscosity subsolution} (resp.\ \emph{supersolution}) of \eqref{e:gsde}
if whenever $\varphi \in \mathit{A}_{P}(U_T)$ and $(\hat{t}, \hat{x}) \in U_T$ satisfy
$$
\max_{U_T}(u^*-\varphi) = (u^*-\varphi)(\hat{t}, \hat{x}) = 0,\qquad
(\hbox{resp.}\quad \min_{U_T}(u_*-\varphi) = (u_*-\varphi)(\hat{t}, \hat{x}) = 0)
$$
then the inequality
$$
\varphi_t(\hat{t}, \hat{x})+F(\Lambda_{W}(\varphi(\hat{t}, \cdot))(\hat{x})) \le 0,\qquad
(\hbox{resp.}\quad \varphi_t(\hat{t}, \hat{x})+F(\Lambda_{W}(\varphi(\hat{t}, \cdot))(\hat{x}))\ge 0)
$$
holds.
% Here $u^*$ (resp. $u_*$) denotes the upper semicontinous (resp. lower semicontinous) envelope of $u$. \marginpar{AN: def. please}
We say that $u$ is a \emph{viscosity solution}
if $u$ is both a viscosity subsolution and a viscosity supersolution.
\end{definition}

We close this subsection with an explicit example 
of a viscosity solution starting from general initial data $u_0$, not necessarily continuous.
%such as $BV$ functions.

\begin{example}
Let us consider the total variation flow \eqref{e:tvf} with the initial datum
$$
u_0(x) =
\begin{cases}
1 & \text{if $x \in (0, a)$,} \\
0 & \text{if $x \in (a, 1)$,}
\end{cases}
$$
where $a \in (0, 1)$ is a given constant.
One will realize that this problem \eqref{e:tvf}, \eqref{e:ic} has a solution of the form
$$
u(t, x) =
\begin{cases}
1-\frac{2}{a}t & \text{if $t \le a(1-a)/2$, $x \in (0, a)$,} \\
\frac{2}{1-a}t & \text{if $t \le a(1-a)/2$, $x \in (a, 1)$,} \\
a & \text{if $t \ge a(1-a)/2$.}
\end{cases}
$$
In fact it is not difficult to check that $u$ is a viscosity solution allowing a discontinuity.
\end{example}

\subsection{Comparison and stability of viscosity solutions}

We recall the comparison and stability results applicable in the present work.

\begin{proposition}{\rm (Comparison, \cite[Theorem 7]{GGR14})}
\label{t:vcp}
We assume (F1) and (F2) and we take $U = \mathbb{T}$.
Let $u$ and $v$ respectively be a viscosity sub- and supersolution of \eqref{e:gsde} such that $u^* < +\infty$ and $v_* > -\infty$.
If $u^*|_{t = 0} \le v_*|_{t = 0}$ on $\bT$, then $u^* \le v_*$ in $Q_T$.
\end{proposition}

\begin{proposition}{\rm (Stability under extremum, \cite[Theorems 4.1 and 5.4]{GGN13})}
\label{t:vsup}
Assume (F1) and (F2).
Let $S$ be a family of viscosity subsolutions (resp.\ supersolutions) of \eqref{e:gsde}.
Then, %the supremum (resp.\ infimum)
$$
u(t, x) := \sup_{v \in S} v(t, x) \qquad(\hbox{resp. } u := \inf_{v \in S} v(t, x))
$$
is a real-valued viscosity subsolution (resp.\ supersolution) of \eqref{e:gsde}.
\end{proposition}

In summary, one is able to construct a unique solution of \eqref{e:gsde} with continuous initial data.

%\begin{definition}[Semi-unifom convergence]
%Let $u_n, u$ are $\bar{\mathbf{R}}$-valued functions.
%We say that \emph{$u_n$ converges to $u$ upper (resp.\ lower) semi-uniformly}
%if they satisfy the following conditions:
%\begin{enumerate}
%\item[(1)]
%For every $x_n \to x$ the inequality $\limsup u_n(x_n) \le u(x)$ (resp.\ $\liminf u_n(x_n) \ge %u(x)$) holds.
%\item[(2)]
%For every $x$ there exists $x_n \to x$ such that the inequality $\liminf u_n(x_n) \ge u(x)$ (resp.%\ $\limsup u_n(x_n) \le u(x)$) holds.
%\end{enumerate}
%\end{definition}

%\begin{proposition}[Stability under limit]
%\label{t:vstab}\todo{Unnecessary. Delete?}
%Assume (F1) and (F2).
%Let $u^k$ be a viscosity subsolution (resp.\ supersolution) of \eqref{e:gsde} for $k = 0, 1, \cdots$.
%Then, the upper (resp.\ lower) semilimit
%$$
%\text{$\overline{u}(t, x) := \limsup_{k \to \infty, t' \to t, x' \to x} u^k(t', x')$}
%$$
%is a viscosity subsolution (resp.\ supersolution) of \eqref{e:gsde}.
%\end{proposition} See \cite[Theorem 1.5]{GG99}.

\subsection{Related work}

Pioneer work for the very singular diffusion equations including \eqref{e:gsde} is given by Giga-Giga in \cite{GG98} and \cite{GG99},
which study the equations with spatially homogeneous external force.
The authors introduced a notion of viscosity solutions,
and established comparison, existence and stability of solutions in the sense of \cite{GG98}.

The definition of solutions in \cite{GG98} looks different from one in the present paper.
However, these two definitions are equivalent as far as one considers the equation without a spatially inhomogeneous external force.
Note that the value of $\Lambda_W$ is determined explicitly when the external force is independent of the spatial variable.
On the other hand, when the external force term depends on the spatial variable, one will encounter an obstacle problem.
By considering the obstacle problem carefully,
a comparison principle is given in \cite{GGR14},
while a general existence result based on Perron method is given in \cite{GGN13}.
% A generalization to equations with an external force term has been done by \cite{GGR14} and \cite{GGN13}.
% In such a case the obstacle problem contains an inhomogeneous term $Z$
% and thus the nonlocal curvature cannot be calculate explicitly.
% By careful observation for the obstacle problem a comparison principle is established in \cite{GGR14} and a construction of a solution via Perron method is given by \cite{GGN13}.

As a further development of the viscosity solutions theory for singular diffusion equations,
we point out that higher-dimensional total variation flows of a non-divergence form is studied in \cite{GGP14}.

\section{Energy solutions are viscosity solutions}

We prove here our main result.

\begin{theorem}
\label{t:ma}
 If $u_0\in BV(\bT)$, then the corresponding energy solution constructed in Theorem \ref{tw-gl} is also a viscosity solution in the sense of Definition \ref{d:v}.
\end{theorem}

The proof is divided in a number of steps.

\subsection{Continuous initial data}

We begin our analysis with the case of continuous initial data $u_0$. Actually, we have seen in Proposition \ref{t:re} that $u_0$, $u_{0,x}\in BV (\bT)$ guarantee that $u \in C^{\alpha,\alpha/2}(\bT\times[0,T])$, for a positive $\alpha$.

\subsection{Preparation of the initial conditions}

%We mentioned that the viscosity theory is better developed for continuous initial conditions.
We prove here that we can approximate the $BV$ data by continuous functions in an appropriate way.

\begin{proposition}
\label{t:pre}
 Let us suppose that $u_0\in BV(\bT)$.  Then, \\
 (a) there exists an increasing sequence $v_0^k$, $k\in \bN$ of continuous functions and such that $v_0^k \le u_0$ a.e. and 
 $$
 \lim_{k\to 0} v_0^k(x) = u_0(x)\qquad \hbox{for } a.e.\ x\in\bT;
 $$
 (b) there exists a decreasing sequence $w_0^k$, $k\in \bN$ of continuous functions and such that $w_0^k\ge u_0$ a.e. and 
 $$
 \lim_{k\to 0} w_0^k(x) = u_0(x)\qquad \hbox{for } a.e.\ x\in\bT.
 $$
\end{proposition}
The proof will be in a few steps.
The first step is the Lemma below.

\begin{lemma}
\label{l:11}
Let us suppose that $u:[a,b)\to \bR$ is increasing, then there exists an increasing sequence of continuous functions $\{u_k\}_{k=1}^\infty$ such that $u_k(x) \le u(x)$ and
$$
\lim_{k\to\infty} u_k(x) = u(x)\qquad\hbox{for }a.e. \ x\in\bT.
$$
\end{lemma}

\begin{proof}
We extend $u$ to $\bR$ by the following formula,
$$
\tilde u(x) =
\left\{
\begin{array}{ll}
 u(b^-), &\hbox{for } x\ge b,\\
 u(x), &\hbox{for }  x\in (a,b),\\
 u(a), &\hbox{for }  x\le a.
\end{array}
\right.
$$
Let us denote  by $\rho$ a standard mollifying kernel, whose support is $[-1,1]$. We set
$$
u_k (x) = (\tilde u *\rho_{1/k})(x-\frac1k).
$$
Obviously, $u_k$ are continuous. We have to check its properties.
We notice, that  supp$\rho\subset[-1,1]$ implies 
$$
u_k (x) = \int_\bR u(x-(y+1)/k) \rho(y)\,dy \le \int_\bR u(x) \rho(y)\,dy =  u(x) .
$$
By the same argument, we have
$$
u_{k+1}(x) \ge u_k (x).
$$

We know that $u*\rho_{1/k}$ converges to $u$ in $L^1$, hence there exists a subsequence, still denoted by $u_k$ converging to $u$ almost everywhere. Since the shift operator, 
$v(\cdot)\mapsto v(\cdot-h)$ is continuous in $L^1$, we conclude that $u_k(x)$ converges to $u(x)$ for almost every $x\in [a,b)$.
\end{proof}

\begin{lemma}\label{l:12}
 Let us suppose that $f:[a,b]\to \bR$ is such that  $f$ restricted to $[a,b)$ is increasing and $f(b) \in [f(a), f(b^-))$. Then, there exists an increasing sequence $\{f_k\}_{k=1}^\infty$ converging almost everywhere to $f$.
\end{lemma}

\begin{proof}
We can take $u= f|_{[a,b)}$. Lemma \ref{l:11} yields an increasing sequence $u_k$ converging a.e. to $u$. We define it by the formula below,
$$
f_k(x) = \left\{
\begin{array}{ll}
 u_k(x) & x\in [a, b-1/k),\\
 \ell_k(x-b) + f(b^-) & x\in [b-1/k,b],
\end{array}
\right.
$$ 
where 
$$
\ell_k = k(f(b^-) - u_k(b-1/k)).
$$
It is easy to see that $f_k$ has the desired properties.
\end{proof}

\begin{lemma}
\label{l:13}
 Let us suppose that $f:[a,b]\to \bR$ is such that  $f$ restricted to $[a,b)$ is decreasing and $f(b) \in [f(b^-), f(a))$. Then, there exists a decreasing sequence $\{f_k\}_{k=1}^\infty$ converging almost everywhere to $f$.
 \end{lemma}

\begin{proof}
This result follows from Lemma \ref{l:12} applied to $-f$.
\end{proof}
 
Now, we are ready to prove Proposition \ref{t:pre}. It is sufficient to establish (a) since (b) will follow from (a) applied to $-u_0$.

If we are given $u_0\in BV(\bT)$, then we may choose its `good representative' see \cite[Theorem 3.28]{AFP00} and we define $f:[0,1]\to \bR$ by the following formula,
$$
f(x) =\left\{
\begin{array}{ll}
 u_0(x) & x \in [0,1)\\
 u_0(0) & x =1.
\end{array}
\right.
$$
Of course, $f$ has a finite variation over $(0,1)$, hence $f = f^+ -f^-$, where $f^\pm$ are increasing. By Lemmas \ref{l:12} and \ref{l:13}, there exist $v^+_k, v^-_k, w^+_k, w^-_k$ such that
$$
v^+_k \le f^+ \le w^+_k, \qquad v^-_k \le f^- \le w^-_k.
$$
We set $v_k = v^+_k - w^-_k$ and  $w_k = w^+_k - v^-_k$, these sequences monotonically converge to $u_0$ almost everywhere. Obviously, by the definition, we have
$$
v_k \le f^+ -f^- \le w_k . \eqno\Box
$$

\subsection{$\mathit{BV}$ initial data}

Let us consider
$v^\epsilon$ a unique  continuous energy solution corresponding to  $v^\epsilon_0$, given by Proposition \ref{t:pre}. Respectively, let
$w^\epsilon$ be a unique continuous energy solution corresponding to  $w^\epsilon_0$, given by Proposition \ref{t:pre}. We set
$$
v = \sup v^\epsilon, \qquad w = \sup  w^\epsilon.
$$
\begin{lemma}
 If $v$ and $w$ are defined above, then $v$ and $w$  are energy solutions.
\end{lemma}

\begin{proof} Due to monotonicity of families $v^\epsilon$ and $w^\epsilon$, we see that $v$ and $w$ are in fact pointwise limits, so that we can use
%We have already shown in 
Proposition \ref{t:sta} implying that
$v$ and $w$ are energy solutions.
Hence, our claim follows.
\end{proof}

\begin{lemma}
$v$ is a viscosity subsolution and $w$ is a viscosity supersolution of \eqref{e:ge}, respectively.
\end{lemma}
\begin{proof}
 This lemma follows immediately from Proposition \ref{t:vsup}.
\end{proof}

{\it Proof of Theorem \ref{t:ma}.} We showed the following inequality for viscosity subsolution $v$ and viscosity supersolution $w$,
$$v\le w.$$ At the same time $v$ and $w$ are energy solutions to (\ref{e:ge}) with the same initial condition $u_0$. Uniqueness of energy solutions implies $v=w$, hence energy solutions to (\ref{e:ge}-\ref{e:ic}) are viscosity solutions to this equation. $\Box $

\begin{remark}
 We constructed in Theorem \ref{tw-gl}  solutions to (\ref{e:ge})--(\ref{e:ic}) by the regularization method for data in $BV$. Then, we showed  in Theorem \ref{t:ma} that these solutions are also viscosity solutions. In 
\cite[Theorem 1]{mury} by the same method, we constructed solutions to (\ref{e:ge})--(\ref{e:ic}) assuming more regular data, i.e. $u_0$ and $u_{0,x}\in BV$ and we showed that they are {\it almost classical}, \cite[Theorem 3]{mury}, see also \cite{mury-cb}. Our  Theorem \ref{t:ma} implies that all these types of solutions coincide, at least for more regular initial conditions. Finally, we point
out that they are in fact solutions in the sense of the nonlinear semigroup theory, i.e. in a weak sense.
\end{remark}

\section*{Acknowledgement}
Both authors thank the grant agencies for partial support during the preparation of this paper and the universities, which hosted them. AN was supported by a Grant-in-Aid for JSPS 
Fellows No. 25-7077 as well as by NCN  through grant 2011/01/B/ST1/01197
and by MNiSW through 2853/7.PR/2013/2 grant. 
PR was in part supported by the EC  IRSES program "Flux" and  by NCN  through grant 2011/01/B/ST1/01197.
A part of the research was performed during AN visits to the University of Warsaw supported by the
Warsaw Center of Mathematics and Computer Science through the program `Guests'. 
Some work was done during a PR visit to the University of Tokyo.

% \bibliographystyle{amsplain}
% \bibliography{references}

\providecommand{\bysame}{\leavevmode\hbox to3em{\hrulefill}\thinspace}
\providecommand{\MR}{\relax\ifhmode\unskip\space\fi MR }
\providecommand{\MRhref}[2]{%
  \href{http://www.ams.org/mathscinet-getitem?mr=#1}{#2}
}
\providecommand{\href}[2]{#2}

\end{document}